\documentclass[reqno]{amsart}
\usepackage{amsmath}
\usepackage[dvips]{graphicx}
\usepackage{amsfonts}
\usepackage{amssymb}
\usepackage{latexsym}
\graphicspath{{Img/}}

\newtheorem{theorem}{Theorem}

\theoremstyle{plain}

\newtheorem{corollary}{Corollary}

\newtheorem{lemma}{Lemma}

\newtheorem{problem}{Problem}
\newtheorem{proposition}{Proposition}
\newtheorem{remark}{Remark}

\numberwithin{equation}{section}

\begin{document}
\title[An $\epsilon$ characterization of the vertices of tetrahedra]{An $\epsilon$ characterization of the vertices of tetrahedra in the three dimensional Euclidean Space}
\author{Anastasios N. Zachos}
\address{Greek Ministry of Education, Athens, Greece}
\email{azachos@gmail.com}

\begin{abstract}
We determine a positive real number (weight), which corresponds to a vertex of a tetrahedron and it depends on the three weights which correspond to the other three vertices and an infinitesimal number $\epsilon.$ As a limiting case, for $\epsilon\to 0$ the quad of the corresponding weights yields a degenerate weighted Fermat-Torricelli tree which coincides with the three neighboring edges of the tetrahedron and intersect at this vertex.

\end{abstract}\maketitle

\section{Introduction}

Let $ A_{1}A_{2}A_{3}A_{4}$ be a tetrahedron in $\mathbb{R}^{3}.$
We denote by $B_{i}$ a non-negative number (weight) which
corresponds to each vertex $A_{i},$ $A_{0}$ be a point in
$\mathbb{R}^{3},$ by $a_{ij}$ be the Euclidean distance of the
linear segment $A_{i}A_{j},$ and by $\vec {u}(A_{i},A_{j})$ is the unit vector with direction from
$A_{i}$ to $A_{j},$ for $i,j=0,1,2,3,4$ and $i\ne j.$

The weighted Fermat-Torricelli problem for
$A_{1}A_{2}A_{3}A_{4}$ in $\mathbb{R}^{3}$ states that:

\begin{problem}\label{5FT}
Find a point $A_{0}$ (weighted Fermat-Torricelli point)
\begin{equation} \label{eq:001}
\sum_{i=1}^{4}B_{i}a_{0i}\to min.
\end{equation}
\end{problem}

Y. Kupitz and H. Martini proved the existence and uniqueness of the weighted
Fermat-Torricelli point for $n$ non-collinear and non-coplanar points in
$\mathbb{R}^{m}$ (\cite{Kup/Mar:97}, \cite{BolMa/So:99}).
Furthermore, a characterization of the weighted Fermat-Torricelli point (weighted floating and absorbed cases) is given
in \cite[Theorem~18.37, p.~250]{BolMa/So:99}).

By setting $n=4$ and $m=3,$ the following theorem is given for a tetrahedron $A_{1}A_{2}A_{3}A_{4}:$

\begin{theorem}\cite[Theorem~18.37, p.~250]{BolMa/So:99})\label{theor}

Let $A_{1}A_{2}A_{3}A_{4}$ be a tetrahedron in $\mathbb{R}^{3}.$

(i)The weighted Fermat-Torricelli point $A_{0}$ of
$A_{1}A_{2}A_{3}A_{4}$ exists and is unique.

(ii) If
\[ \|\sum_{j=1}^{4}B_{j}\vec {u}(A_{i},A_{j})\|>B_i, i\neq j, \] for
{i,j}={1,2,3,4}, then

(a) the weighted Fermat-Torricelli point does not belong in $\{A_{1},A_{2},A_{3},A_{4}\}$ (Weighted Floating Case), \\

(b) \[\sum_{i=1}^{4}B_{i}\vec{u}(A_{0},A_{i})=\vec{0}\]

(Weighted balancing condition).

(iii) If there is some i with \[ \|{\sum_{j=1}^{4}B_{j}\vec
u(A_i,A_j)}\|\le B_i, i\neq j. \] for {i,j}={1,2,3,4}, then the
weighted Fermat-Torricelli point is the vertex $A_{i}$

(Weighted Absorbed Case).
\end{theorem}

The unique solution of the weighted Fermat-Torricelli problem for $A_{1}A_{2}A_{3}A_{4}$ is called a weighted Fermat-Torricelli tree
having one node (weighted Fermat-Torricelli point) with degree four (weighted floating case). A numerical approach to find the weighted Fermat-Torricelli tree for $A_{1}A_{2}A_{3}A_{4},$ by introducing a method of differentiation of the length of a linear segment to a specific dihedral angle is given in \cite{Zach/Zou:09}.

The inverse weighted Fermat-Torricelli problem for tetrahedra in
$\mathbb{R}^{3}$ states that:

\begin{problem}
Given a point $A_{0}$ which belongs to the interior of
$A_{1}A_{2}A_{3}A_{4}$ in $\mathbb{R}^{3}$, does there exist a
unique set of positive weights $B_{i},$ such that
\begin{displaymath}
 B_{1}+B_{2}+B_{3}+B_{4} = c =const,
\end{displaymath}
for which $A_{0}$ minimizes
\begin{displaymath}
 f(A_{0})=\sum_{i=1}^{4}B_{i}a_{0i}.
\end{displaymath}

\end{problem}

The unique solution of the inverse weighted Fermat-Torricelli problem for tetrahedra in $\mathbb{R}^{3}$ is given in \cite{Zach/Zou:09}.
We denote by $\alpha_{i,j0k}$ the angle that is formulated by the line segment that connects $A_{0}$ with the trace of the orthogonal projection of $A_{i}$
to the plane defined by $\triangle A_{j}A_{0}A_{k}$ and we set $\alpha_{lmn}\equiv \angle A_{l}A_{m}A_{n},$ for $j,k=1,2,3,4$ and $l,m,n=0,1,2,3,4.$
\begin{proposition}\cite[Proposition~1, Solution of Problem~2]{Zach/Zou:09}\label{propo5}
The weight $B_{i}$ are uniquely determined by the formula:
\begin{equation}\label{inverse111}
B_{i}=\frac{C}{1+\|\frac{\sin{\alpha_{i,k0l}}}{\sin{\alpha_{j,k0l}}}\|+\|\frac{\sin{\alpha_{i,j0l}}}{\sin{\alpha_{k,j0l}}}\|+\|\frac{\sin{\alpha_{i,k0j}}}{\sin{\alpha_{l,k0j}}}\|},
\end{equation}
for $i,j,k,l=1,2,3,4$ and $i \neq j\neq k\neq l.$
\end{proposition}

We need the following two formulas which have been derived in \cite{Zach/Zou:09} and \cite{Zachos:20}, respectively:

\begin{lemma}\cite{Zach/Zou:09}
The ratio $\frac{B_{j}}{B_{i}}$ is given by:
\begin{equation}\label{ratioji}
\frac{B_{j}}{B_{i}}=\sqrt{ \| \frac{\sin^{2}\alpha_{k0m}-\cos^{2}\alpha_{m0i}-\cos^{2}\alpha_{k0i}+2\cos\alpha_{k0m}\cos\alpha_{m0i}\cos\alpha_{k0i}}{
\sin^{2}\alpha_{k0m}-\cos^{2}\alpha_{m0j}-\cos^{2}\alpha_{k0j}+2\cos\alpha_{k0m}\cos\alpha_{m0j}\cos\alpha_{k0j}} \| }
\end{equation}
for $i,j,k,m=1,2,3,4.$
\end{lemma}

\begin{lemma}\cite[Proposition~1]{Zachos:20}
The angles $\alpha_{i,k0m}$ depend on exactly five given angles
$\alpha_{102},$ $\alpha_{103},$ $\alpha_{104},$ $\alpha_{203}$ and
$\alpha_{204}.$ 

The sixth angle $\alpha_{304}$ is calculated by the following formula:

\begin{eqnarray}\label{calcalpha3042}
&&\cos\alpha_{304}=\frac{1}{4} [4 \cos\alpha _{103} (\cos\alpha
_{104}-\cos\alpha _{102} \cos\alpha _{204})+\nonumber\\
&&+2 \left(b+2 \cos\alpha _{203} \left(-\cos\alpha _{102}
\cos\alpha _{104}+\cos\alpha _{204}\right)\right)] \csc{}^2\alpha
_{102}\nonumber\\
\end{eqnarray}

where

\begin{eqnarray}\label{calcalpha304auxvar}
b\equiv\sqrt{\prod_{i=3}^{4}\left(1+\cos\left(2 \alpha
_{102}\right)+\cos\left(2 \alpha _{10i}\right)+\cos\left(2 \alpha
_{20i}\right)-4 \cos\alpha _{102} \cos\alpha _{10i} \cos\alpha
_{20i}\right)}\nonumber\\.
\end{eqnarray}

for $i,k,m=1,2,3,4,$ and $i \ne k \ne m.$
\end{lemma}

If $B_{1},$ $B_{2},$ $B_{3},$ $B_{4}$ satisfy the inequalities of the floating case, we derive
a unique weighted Fermat-Torricelli tree $\{A_{1}A_{0}, A_{2}A_{0}, A_{3}A_{0}, A_{4}A_{0}\},$
which consists of four branches $A_{i}A_{0},$ for $i=1,2,3,4$ and they intersect at $A_{0}.$

If $B_{1},$ $B_{2},$ $B_{3}$ $B_{4}$ satisfy one of the inequalities of the absorbed case, we obtain
a degenerate weighted Fermat-Torricelli tree $\{A_{2}A_{1},A_{1}A_{3},A_{1}A_{4}\},$ $\{A_{1}A_{2},A_{2}A_{3},A_{2}A_{4}\}$
and $\{A_{1}A_{3},A_{3}A_{2},A_{3}A_{4}\},$ $\{A_{1}A_{4},A_{2}A_{4},A_{3}A_{4}\}.$


Assume that:

 \[ \|{\sum_{j=1}^{4}B_{j}\vec
u(A_{4},A_{j})}\|\le B_{4}, i\neq j. \]

Suppose that we choose $B_{1},$ $B_{2},$ $B_{3},$ such that

 \[ \|{\sum_{j=1}^{4}B_{j}\vec
u(A_{4},A_{j})}\|= B_{4}, i\neq j, \]

or

\begin{equation}\label{bb4}
B_{4}^2=B_{1}^2+B_{2}^2+B_{3}^2+2B_{1}B_{2}\cos\alpha_{142}+ 2B_{1}B_{3}\cos\alpha_{143}+2B_{2}B_{3}\cos\alpha_{243}
\end{equation} 
 or

\[B_{4}=f(B_{1},B_{2,}B_{3}).\]

Thus, we consider the following problem:

\begin{problem}
How can we determine the values of $B_{1},$ $B_{2},$ $B_{3},$ such that $f(B_{1}, B_{2}, B_{3})$ gives the minimum value of
$B_{4}$ that corresponds to the vertex $A_{4}$ ?
\end{problem}

If we set $B_{1}=0$ or $B_{2}=0$ or $B_{3}=0$ in Problem~1, $B_{4}$ depends of the values of two weights which are determined 
in \cite{Zachos:20b}.  




In this paper, we determine the value of $B_{4}$ by introducing an infinitesimal real number $\epsilon,$ ($\epsilon$  characterization of $A_{4}$)
such that:
$\alpha_{120}=\alpha_{124}-\epsilon,$ $\alpha_{102}= \alpha_{142}+k_{142} \epsilon,$
$\alpha_{203}= \alpha_{243}+k_{203}\epsilon,$ $\alpha_{103}=\alpha_{143}+k_{143}\epsilon$ and $\alpha_{023}=\alpha_{423}+k_{423}\epsilon,$ where $k_{142},$ $k_{203},$ $k_{143}$ and $k_{423}$ are   rational numbers,  by applying the solution of the inverse weighted
Fermat-Torricelli problem for tetrahedra in $\mathbb{R}^{3}.$


\section{An $\epsilon$ characterization of the vertices of tetrahedra in $\mathbb{R}^{3}$}

Let $A_{0}$  be an interior point of $A_{1}A_{2}A_{3}A_{4}$ in $\mathbb{R}^3.$

We denote by $h_{0,12}$ the length of the height of $\triangle
A_{0}A_{1}A_{2}$ from $A_{0}$ with respect to $A_{1}A_{2},$  by $\alpha,$ the dihedral angle which is formed
between the planes defined by $\triangle A_{1}A_{2}A_{3}$ and
$\triangle A_{1}A_{2}A_{0},$ and by $\alpha_{g_{4}}$ the dihedral
angle which is formed by the planes defined by $\triangle
A_{1}A_{2}A_{4}$ and $\triangle A_{1}A_{2}A_{3}.$

We set
$\alpha_{120}=\alpha_{124}-\epsilon,$ $\alpha_{102}= \alpha_{142}+k_{142} \epsilon,$
$\alpha_{203}= \alpha_{243}+k_{203}\epsilon,$ $\alpha_{103}=\alpha_{143}+k_{143}\epsilon$ and $\alpha_{023}=\alpha_{423}+k_{423}\epsilon,$ where $k_{142},$ $k_{203},$ $k_{143}$ and $k_{423}$ are rational numbers.

We need the following two formulas given by the following proposition which is derived in \cite{Zach/Zou:09}
and expresses $a_{03}$ and $a_{04}$ as a function w.r. to $a_{01},$ $a_{02}$ and $\alpha.$

\begin{proposition}\cite{Zach/Zou:09}\label{a03a04}
The length $a_{0i}$ depends on $a_{01}, a_{02}$ and $\alpha :$

\begin{equation}\label{eq:deral2}
a_{0i}^2=a_{02}^2+a_{2i}^2-2a_{2i}[\sqrt{a_{02}^2-h_{0,12}^2}\cos({\alpha_{12i}})+h_{0,12}\sin({\alpha_{12i}})\cos({\alpha_{g_{i}}}-\alpha)],
\end{equation}

or

\begin{equation}\label{eq:derall2n}
a_{0i}^2=a_{01}^2+a_{1i}^2-2a_{1i}[\sqrt{a_{01}^2-h_{0,12}^2}\cos({\alpha_{21i}})+h_{0,12}\sin({\alpha_{21i}})\cos({\alpha_{g_{i}}}-\alpha)].
\end{equation}

where

\begin{equation}\label{height012}
h_{0,12}=\frac{a_{01}a_{02}}{a_{12}}\sqrt{1-\left(\frac{a_{01}^{2}+a_{02}^{2}-a_{12}^2}{2a_{01}a_{02}}\right)^{2}}
\end{equation}

for $i=3,4.$ 
\end{proposition}

By replacing the index $0\to 4$ for $i=3$ in (\ref{eq:deral2}), we derive the following corollary:

\begin{corollary}\label{alphag4}
\begin{equation}\label{ag4}
\alpha_{g_{4}}=\arccos(\frac{a_{42}^2+a_{23}^2-a_{43}^2-2a_{23}\sqrt{a_{42}^2-h_{4,12}^2}\cos({\alpha_{213}})}{2a_{23}h_{4,12}\sin\alpha_{123}})
\end{equation}

where

\begin{equation}\label{height412}
h_{4,12}=\frac{a_{41}a_{42}}{a_{12}}\sqrt{1-\left(\frac{a_{41}^{2}+a_{42}^{2}-a_{12}^2}{2a_{41}a_{42}}\right)^{2}}.
\end{equation}

\end{corollary}

By solving (\ref{eq:deral2}) w.r. to $\alpha$ for $i=3$ and by replacing the derived formula (\ref{eq:deral2}) for $i=4$
and taking into account (\ref{ag4}), we derive that $a_{04}=a_{04}(a_{01},a_{02},a_{03}).$

\begin{corollary}\label{coroa04}
\begin{equation}\label{a04a01a02a03}
a_{04}^2=a_{02}^2+a_{24}^2-2a_{24}[\sqrt{a_{02}^2-h_{0,12}^2}\cos({\alpha_{124}})+h_{0,12}\sin({\alpha_{124}})(\cos({\alpha_{g_{4}}})\cos\alpha+\sin({\alpha_{g_{4}}})\sin\alpha)]
\end{equation}

where

\begin{equation}\label{alpha}
\alpha=\arccos(\frac{a_{02}^2+a_{23}^2-a_{03}^2-2a_{23}\sqrt{a_{02}^2-h_{0,12}^2}\cos({\alpha_{213}})}{2a_{23}h_{0,12}\sin\alpha_{123}}).
\end{equation}

\end{corollary}

\begin{proposition}\label{a04epsilon}
If $\alpha_{120}=\alpha_{124}-\epsilon,$ $\alpha_{102}= \alpha_{142}+k_{142} \epsilon,$
$\alpha_{203}= \alpha_{243}+k_{243}\epsilon,$ $\alpha_{103}=\alpha_{143}+k_{143}\epsilon$ and $\alpha_{023}=\alpha_{423}+k_{423}\epsilon,$
then $a_{04}=a_{04}(\epsilon).$
\end{proposition}

\begin{proof}
By applying the law of sines in $\triangle A_{1}A_{0}A_{2},$
we get:

\begin{equation}\label{a01e}
a_{01}(\epsilon)=\frac{\sin(\alpha_{124}-\epsilon)a_{12}}{\sin(\alpha_{142}+k_{142} \epsilon)}.
\end{equation}

and

\begin{equation}\label{a02e}
a_{02}(\epsilon)=\frac{\sin(\alpha_{124}+(k_{142}-1)\epsilon+\alpha_{142})a_{12}}{\sin(\alpha_{142}+k_{142} \epsilon)}.
\end{equation}

By applying the law of sines in $\triangle A_{2}A_{0}A_{3},$
we get:

\begin{equation}\label{a03e}
a_{03}(\epsilon)=\frac{\sin(\alpha_{423}+k_{423}\epsilon)a_{23}}{\sin(\alpha_{243}+k_{243} \epsilon)}.
\end{equation}

By replacing (\ref{a01e}), (\ref{a02e}), (\ref{a03e}) in (\ref{a04a01a02a03})
we obtain that\\
 $a_{04}=a_{04}(a_{01}(\epsilon),a_{02}(\epsilon),a_{03}(\epsilon))=a_{04}(\epsilon).$

\end{proof}



\begin{theorem}\label{echarR3}
The weight $B_{i}=B_{i}(\epsilon)$ are uniquely determined by the formula:
\begin{equation}\label{inverse111e}
B_{i}=\frac{C}{1+\|\frac{\sin{\alpha_{i,k0l}}}{\sin{\alpha_{j,k0l}}}\|+\|\frac{\sin{\alpha_{i,j0l}}}{\sin{\alpha_{k,j0l}}}\|+\|\frac{\sin{\alpha_{i,k0j}}}{\sin{\alpha_{l,k0j}}}\|},
\end{equation}

where

\begin{equation}\label{ratioije}
\|\frac{\sin{\alpha_{i,k0m}}}{\sin{\alpha_{j,k0m}}}\|=\sqrt{ \| \frac{\sin^{2}\alpha_{k0m}-\cos^{2}\alpha_{m0i}-\cos^{2}\alpha_{k0i}+2\cos\alpha_{k0m}\cos\alpha_{m0i}\cos\alpha_{k0i}}{
\sin^{2}\alpha_{k0m}-\cos^{2}\alpha_{m0j}-\cos^{2}\alpha_{k0j}+2\cos\alpha_{k0m}\cos\alpha_{m0j}\cos\alpha_{k0j}} \| }
\end{equation}

depends on $\epsilon,$ $\alpha_{102}(\epsilon),$ $\alpha_{103}(\epsilon),$ $\alpha_{203}(\epsilon),$ and $\alpha_{204}(\epsilon),$
such that there exists a functional dependence between the rational numbers $k_{142},$ $k_{203},$ $k_{143}$ and $k_{423},$

\begin{equation}\label{fd}
F(k_{142},k_{203},k_{143},k_{423})=0,
\end{equation}

for $i,j,k,l=1,2,3,4$ and $i \neq j\neq k\neq l.$

\end{theorem}

\begin{proof}
By replacing (\ref{a02e}) and $a_{04}(\epsilon)$ in the law of cosines in $\triangle A_{2}A_{0}A_{4},$
we obtain :

\begin{equation}\label{cosalpha204}
\alpha_{204}(\epsilon)=\arccos(\frac{(a_{02}(\epsilon))^2+(a_{04}(\epsilon))^2-a_{24}^2}{2a_{02}(\epsilon)a_{04}(\epsilon)}).
\end{equation}

By replacing (\ref{a03e}) and $a_{04}(\epsilon)$ in the law of cosines in $\triangle A_{3}A_{0}A_{4},$
we obtain :

\begin{equation}\label{cosalpha304}
\alpha_{304}(\epsilon)=\arccos(\frac{(a_{03}(\epsilon))^2+(a_{04}(\epsilon))^2-a_{34}^2}{2a_{03}(\epsilon)a_{04}(\epsilon)}).
\end{equation}

By replacing (\ref{cosalpha204}), $\alpha_{102}= \alpha_{142}+k_{142} \epsilon,$
$\alpha_{203}= \alpha_{243}+k_{243}\epsilon,$ $\alpha_{103}=\alpha_{143}+k_{143}\epsilon$ in (\ref{calcalpha3042}) and by substituting the derived result in the left side of (\ref{cosalpha304}), we derive a functional dependence $F(k_{142},k_{203},k_{143},k_{423})=0.$
By replacing the five angles $\alpha_{102}(\epsilon),$ $\alpha_{103}(\epsilon),$ $\alpha_{104}(\epsilon),$ $\alpha_{203}(\epsilon),$
and $\alpha_{204}(\epsilon),$ for three given rational numbers  $k_{142},$ $k_{203},$ $k_{143}$ and the fourth rational $k_{423}$ is determined by (\ref{fd}) in (\ref{inverse111}), we obtain that the weights $B_{i}(\epsilon)$ are determined by (\ref{inverse111e}).

\end{proof}

\begin{corollary}
For $\epsilon\to 0,$ we derive a degenerate weighted Fermat-Torricelli tree $\{A_{1}A_{4}, A_{2}A_{4}, A_{3}A_{4}\}.$
\end{corollary}

\begin{remark}
By substituting $B_{1}(\epsilon),$ $B_{2}(\epsilon),$ $B_{3}(\epsilon)$ in (\ref{bb4}), we obtain
$B_{4}:$

\[B_{4}=\]
\[=\sqrt{B_{1}(\epsilon)^2+B_{2}(\epsilon)^2+B_{3}(\epsilon)^2+2B_{1}(\epsilon)B_{2}(\epsilon)\cos\alpha_{142}+ 2B_{1}(\epsilon)B_{3}(\epsilon)\cos\alpha_{143}+2B_{2}(\epsilon)B_{3}(\epsilon)\cos\alpha_{243}}.\]
Thus, $\|B_{4}-B_{4}(\epsilon)\|$ gives an $\epsilon$ approximation of the optimal value of $B_{4},$ that achieves the vertex $A_{4}.$
\end{remark}

\end{document}